\documentclass[12pt]{article}
\usepackage{amssymb,latexsym,amsmath,amsthm}
\usepackage{mathrsfs}
\renewcommand{\epsilon}{\varepsilon}
\newcommand{\bbR}{\mathbb{R}}
\newcommand{\GG}{\mathcal{G}}
\newcommand{\half}{{\textstyle \frac{1}{2}}}
\newcommand{\rd}{\mathrm{d}}

\newtheorem{theorem}{Theorem}
\newtheorem{lemma}{Lemma}

\newtheorem{definition}{Definition}

\begin{document}

\noindent
\begin{center}
{\Large \bf Nonlinear stability for steady vortex pairs}
\end{center}
\begin{center}
Geoffrey R. Burton, Helena J. Nussenzveig Lopes\\ and Milton C.
Lopes Filho\\
\end{center}

\bigskip

\begin{center}{\bf Abstract}\end{center}

In this article, we prove nonlinear orbital stability for steadily
translating vortex pairs, a family of nonlinear waves that are exact
solutions of the incompressible, two-dimensional Euler equations.
We use an adaptation of Kelvin's variational principle, maximizing
kinetic energy penalised by a multiple of momentum among mirror-symmetric
isovortical rearrangements.
This formulation has the advantage that the functional to be maximized and
the constraint set are both invariant under the flow of the time-dependent
Euler equations, and this observation is used strongly in the analysis.
Previous work on existence yields a wide class of examples to which our
result applies.

\section{Introduction.}

From a mathematical viewpoint, steady vortex pairs are a class of nonlinear
waves, travelling wave solutions of the incompressible, two dimensional Euler
equations in the full plane.
Two special examples are Lamb's circular vortex-pair, see \cite[p. 245]{Lamb},
and a pair of point vortices with equal magnitude and opposite signs.  

The literature of vortex pairs goes back to the work of Pocklington in
\cite{Pocklington},
with contemporary interest beginning from the work of Norbury, Deem \& Zabusky
and Pierrehumbert,
see \cite{D:Z,Norbury,Pierrehumbert}. The existence (and abundance) of steady vortex pairs
has been rigorously established in two different ways, as a nonlinear eigenvalue problem, see
\cite{Yang91,Norbury} or by optimization in rearrangement classes, see
\cite{GRB:VP,GRB:MCF}. The literature 
on vortex pairs includes asymptotic studies, see \cite{YK94}, numerical studies, see \cite{pullin92} 
and experimental work, see \cite{D:S}. Some analytical results (see \cite{Moffatt}) and numerical evidence, 
\cite{oz82}, suggest orbital stability of steady vortex pairs under appropriate conditions. Still, this 
stability has been an interesting open problem, see \cite{saffman95}.      

Vortex pairs are one instance of a large collection of coherent 
structures found in two dimensional vortex dynamics, for example, 
single vortices, co-rotating pairs and vortex streets. 
In the stability theory of such structures, there has long been a view, growing
out of ideas of Kelvin \cite{Kelvin}, that steady fluid flows representing
extrema of kinetic energy relative to an ``isovortical surface'' are stable;
this viewpoint is exemplified by the formal arguments of Arnol$^\prime$d
\cite{VIA:COND} and informs the variational principles for steady vortex-rings
in three dimensions proposed by Benjamin \cite{TBB}, which provides the impetus
for our work.

The present paper is a piece of real analysis proving a theorem of this type,
the first one that applies to steady vortex pairs. Vortex pairs can be viewed equivalently
as the dynamics of vorticity which is odd with respect to a straight line or as general
vortex dynamics on a half plane, see \cite{L2X01} for a full discussion.
For convenience, we formulate our analysis in terms of steady vortices in a
uniform flow in the half-plane, which corresponds in the full plane to
stability under symmetric perturbations.
Maximizers of a linear combination of the classically preserved quantities
of kinetic energy and impulse over all vorticities that are equimeasurable
rearrangements 
of a fixed non-negative function with bounded support are considered.
The argument is not along the lines envisaged by Arnol$^\prime$d, but 
is analogous to that used in \cite{GRB:stab} for bounded
planar domains; the velocity field of a flow with nearby initial
vorticity is used to convect the steady state and the differences in
energy are estimated.

The vorticity is assumed to be in $L^p \cap L^1$ for some $p>2$ and a
distance between vorticity fields is defined in terms of the $2$-norm and
the impulse.
Because of the invariance under translations parallel to the edge of the
half-plane, maximizers will not be isolated, so our notion of stability is
one of {\em orbital stability}, in which solutions starting close to the
set of maximizers remain close.
These results allow discontinuous vorticity, and the solutions studied are
known not in closed form, but rather via existence theory.
Some stability results in a more symmetric setting, also allowing discontinuous
vorticity, have been given by Marchioro \& Pulvirenti \cite{CM:MP} and
Wan \& Pulvirenti \cite {YHW:MP}.
Precise definitions and formulations of the theorems are given in Section
\ref{Defs}.

Methodologically, a major difficulty is loss of compactness caused by the
unbounded domain of the flow; this is overcome using a
concentration-compactness argument.

\section{Notation and Definitions.}\label{Defs}
We denote by $\Pi$ the half-plane
\[
\Pi = \{ x=(x_1,x_2) \in \bbR^2 \mid x_2>0 \}.
\]
Let $\GG$ denote the inverse for $-\Delta$ in $\Pi$, given by
\begin{equation}
\GG \xi (x) = \int_\Pi G(x,y) \xi (y) \rd y,
\label{eq2}
\end{equation}
whenever this integral converges; here $G$ is the Green's function
given by
\[
G(x,y) = \frac{1}{4\pi} \log \left( \frac{(x_1 -y_1)^2 + (x_2 +
y_2)^2}{(x_1-y_1)^2 + (x_2 - y_2)^2} \right).
\]
It is shown in \cite{GRB:Lamb} that finiteness of
$\|\xi\|_X:=\|\xi\|_2+I(|\xi|)$
is sufficient for convergence of the integral in (\ref{eq2}), where $I$
is defined below.

The kinetic energy due to vorticity $\xi$ is then given by
\[
E(\xi) = \frac{1}{2} \int_\Pi \xi(x) \GG\xi(x) \rd x
\]
and the impulse of linear momentum in the $x_1$-direction is given
by
\[
I(\xi) = \int_\Pi \xi(x_1,x_2) x_2 \rd x .
\]
It is shown in \cite{GRB:Lamb} that $E$ is continuous with respect to
$\|\cdot\|_X$.
We also make use of $\|\xi\|_Y:=\|\xi\|_2+|I(\xi)|$, which is a non-equivalent,
and incomplete, norm on $X$.
The Lebesgue measure, of appropriate dimension, of a measurable set
$\Omega$ is denoted $|\Omega|$.

The evolution of vorticity $\omega$ is governed by the weak form of the
vorticity equation
\begin{equation}
\begin{cases}
\partial_t \omega + \mathrm{div}( \omega u ) = 0, & \\
u = \lambda e_1 + \nabla^\perp \GG\omega,  & (x,t) \in \Pi \times \bbR, \\
\end{cases}
\label{eq1}
\end{equation}
where $\lambda e_1$ represents the velocity of the fluid at infinity,
which is a uniform flow parallel to the $x_1$-axis and
$\nabla^\perp = (-\partial_{x_2},\partial_{x_1})$; the stream function is
then $-\lambda x_2+\GG\omega(x)$.

If $\xi$ is a non-negative Lebesgue integrable function on $\Pi$,
then $\mathcal{R}(\xi)$, the set of {\em rearrangements} of $\xi$ on
$\Pi$, is defined by
\[
\mathcal{R}(\xi) =
\left\{ 0\leq \zeta \in L^1(\Pi) \mbox{ s.t. } \forall \alpha > 0
\; |\{ x:\zeta(x) > \alpha \}| = |\{ x:\xi(x) > \alpha \}|    \; \right\} .
\]
We also define
\[
\mathcal{R}_+(\xi) = \left\{ \zeta 1_\Omega \mid \zeta \in \mathcal{R}(\xi), \;
\Omega \subset \Pi \mbox{ measurable } \right\} .
\]
This is larger than the class $\mathcal{RC}(\xi)$ of
{\em curtailments of rearrangements} defined by Douglas \cite{RJD:UD} as
\[
\mathcal{RC}(\xi) = \left\{ 0 \leq \eta \in L^1(\Pi) \mid
\eta^\Delta = \xi^\Delta 1_{[0,A)}
\mbox{ for some } 0 \leq A \leq \infty \right\},
\]
where $^\Delta$ denotes decreasing rearrangement onto $[0,\infty)$.
We have, from the definitions,
\begin{equation}
\mathcal{R}(\xi) \subset \mathcal{RC}(\xi) \subset
\mathcal{R}_+(\xi) \subset \overline{\mathcal{R}(\xi)^\mathrm{w}} ,
\label{eq14}
\end{equation}
where $\overline{\mathcal{R}(\xi)^\mathrm{w}}$ denotes the closure
of $\mathcal{R}(\xi)$ in the weak topology of $L^2(\Pi)$, this last inclusion
requiring additionally $\xi \in L^2(\Pi)$.
Moreover $\overline{\mathcal{R}(\xi)^\mathrm{w}}$ is convex, see \cite{RJD:UD}.

For example, in the case of a vortex patch, i.e. $\xi = 1_{\Omega}$, where $\Omega$ is a bounded measurable 
subset of the half-plane, we have $\mathcal{R}(\xi)$ is the set of all characteristic functions of sets with the 
same measure as $\Omega$, $\mathcal{RC}(\xi)$ is the set of characteristic functions of sets with measure less than
or equal to the measure of $\Omega$, which is the same as $\mathcal{R}_+(\xi)$. The set 
$\overline{\mathcal{R}(\xi)^\mathrm{w}}$ is much larger, a convex set containing, in particular, functions 
bounded by $1$ which are not piecewise constant.

The {\em (strong) support} $\mbox{suppt}f$ of a real measurable
function $f$ on $\Pi$ is defined to be the set of points of
Lebesgue density 1 for the set $\{ x \in \Pi \mid f(x) \neq
0 \}$\, and is independent of the choice of representative for $f$.

Our stability results are expressed in terms of $L^p$-\emph{regular solutions}
of the vorticity equation, defined below.

\noindent
\begin{definition} \label{LpRegSol}
By an $L^p$\emph{-regular solution} of the vorticity equation we mean 
$\zeta \in L^\infty_\mathrm{loc}([0,\infty), L^1(\Pi)) \cap
L^\infty_\mathrm{loc}([0,\infty), L^p(\Pi))$ satisfying, in the sense of distributions,
\begin{equation} \label{lambdavorteq}
\begin{cases}
\partial_t \zeta + \mathrm{div}( \zeta u ) = 0, & \\
u = \lambda e_1 + \nabla^\perp \GG\zeta,  & (x,t) \in \Pi \times \bbR, \\
\end{cases}
\end{equation}
such that $E(\zeta(t,\cdot))$ and $I(\zeta(t,\cdot))$ are constant.
\end{definition}

Existence of a smooth solution of the initial-boundary-value problem for 
\eqref{lambdavorteq} can be obtained by considering the auxiliary problem
\begin{equation} \label{lambdaVeqn}
\begin{cases}
\partial_t v + (v\cdot\nabla)v + \lambda \partial_{x_1} v = -\nabla p, & \\
\mathrm{div } \,v = 0 & \\
|v| \to 0 & \mbox{ as } |x| \to  \infty, (x,t) \in \Pi \times \bbR\\
\end{cases}
\end{equation}
and taking $\zeta = \mathrm{curl } \,v$. Indeed, taking the curl of \eqref{lambdaVeqn} 
leads to \eqref{lambdavorteq} with $u = v + \lambda e_1$. 

Now, existence of a smooth solution for \eqref{lambdaVeqn}, when the initial vorticity is compactly supported, is a trivial adaptation
of the analogous result for the $2$D incompressible Euler equations, see \cite[Chapter 3]{AJM:ALB},
given that the $L^2$-norm of $v$ is a conserved quantity under evolution by \eqref{lambdaVeqn}. 
Once smooth existence has been established, standard weak convergence methods yield 
existence of weak $L^p$  solutions, see \cite[Theorem 2.1]{LNT00}, again assuming the initial vorticity has compact support. The only remaining issue, to obtain an $L^p$-regular solution for compactly supported initial vorticities, 
is whether $E$ and $I$ are conserved for these weak solutions; this will be the case, easily,
if $p>2$ since, in this case, $v$ is bounded {\it a priori} in $L^r$, $2\leq r \leq \infty$, in 
terms of $L^p$ and $L^1$-norms of vorticity. We note, in particular,
that $L^{\infty}$-regular solutions with compactly supported vorticity are unique, by an
easy adaptation of the celebrated work of Yudovich, see \cite{VIY}. Moreover, these 
$L^{\infty}$-regular solutions are constant along particle paths associated 
with the flow $u$.
Our results do not, however, rely on uniqueness.

Our main result is as follows:

\begin{theorem} \label{thm1} {\bf (Stability Theorem.)}
Let $\zeta_0$ be a non-negative function whose support has finite
positive area $\pi a^2$ ($a>0$) in the half-plane $\Pi$. Suppose
$\zeta_0 \in L^p(\Pi)$, for some $2<p\leq\infty$, and suppose
$\lambda>0$.
Let $\Sigma_\lambda$ denote the set of maximizers of $E-\lambda I$
relative to $\overline{\mathcal{R}(\zeta_0)^\mathrm{w}}$, and
suppose $\emptyset\neq\Sigma_\lambda\subset\mathcal{R}(\zeta_0)$.
Then $\Sigma_\lambda$ is orbitally stable, in the sense that, for
every $\epsilon> 0$ and $A>\pi a^2$, there exists $\delta>0$ such
that, if $\omega(0)\geq 0$ satisfies
$\mbox{dist}_Y(\omega(0),\Sigma_\lambda)<\delta$ and
$|\mbox{suppt}(\omega(0))|<A$, then, for all $t\in\bbR$, we have
$\mbox{dist}_2(\omega(t),\Sigma_\lambda)<\epsilon$, whenever
$\omega(t)$ denotes an $L^p$-regular solution of (\ref{eq1}) with initial data
$\omega(0)$.
\end{theorem}

\noindent
Theorem \ref{thm1} is an analogue, for unbounded domains, of
\cite[Theorem 1]{GRB:stab}, and is deduced, by a similar argument, from
the following result:

\begin{theorem} \label{thm2} {\bf (Maximization Theorem.)}
Let non-negative $\zeta_0 \in L^p(\Pi)$, for some $2<p\leq\infty$ and 
suppose $| \mbox{suppt}(\zeta_0) | = \pi a^2$ for some $0<a<\infty$.
Let $0<\lambda<\infty$ and suppose that the set $\Sigma_\lambda$ in
which $E-\lambda I$ attains its supremum $S_\lambda$ relative to
$\overline{\mathcal{R}(\zeta_0)^\mathrm{w}}$ satisfies
$\Sigma_\lambda \subset \mathcal{R}(\zeta_0)$.
Then, in the context of maximizing $E-\lambda I$ relative to
$\overline{\mathcal{R}(\zeta_0)^\mathrm{w}}$, we have:

\noindent{\rm (i)} every maximizing sequence
comprising elements of
$\mathcal{R}_+(\zeta_0)$ contains a subsequence whose elements,
after suitable translations in the $x_1$-direction, converge in
$\|\cdot\|_2$ to an element of $\Sigma_\lambda$;

\noindent{\rm (ii)} every maximizing sequence
$\{\zeta_n\}_{n=1}^\infty$
comprising elements of $\mathcal{R}_+(\zeta_0)$ satisfies
$\mbox{dist}_2(\zeta_n,\Sigma_\lambda)\to0$ as $n\to\infty$;

\noindent{\rm (iii)} $\Sigma_\lambda$ is non-empty;

\noindent{\rm (iv)} each element $\zeta$ of $\Sigma_\lambda$ is
a translate of a function Steiner-symmetric about the $x_2$ axis,
compactly supported and satisfies
$\zeta = \varphi \circ (\GG\zeta - \lambda x_2)$ a.e. in $\Pi$ for some
increasing function $\varphi$.
\end{theorem}

\noindent {\bf Remarks.} The hypotheses of Theorem \ref{thm1}
exclude $0$ as a maximizer, and therefore the supremum is positive.

We also show in Lemma \ref{lm7} that given non-negative
$\zeta_0 \in L^p(\Pi)$, $p>2$, having compact support, there exists
$\Lambda>0$ such that if $0<\lambda<\Lambda$ then
$\emptyset \neq \Sigma_\lambda \subset \mathcal{R}(\zeta_0)$, so that the
hypotheses of Theorem \ref{thm1} are satisfied.
Lemma \ref{lm7} incidentally provides a mild improvement on the existence
result \cite[Theorem 16(i)]{GRB:VP}.

It also follows from Theorem \ref{thm2}(i) that $\Sigma_\lambda$ comprises a
compact set of functions in $L^2(\Pi)$ together with their $x_1$-translations.

Theorem \ref{thm2}(iv) proves that the maximizers have compact support and
therefore fit into the context of \cite[Theorem 16(i)]{GRB:VP}, which yields
the conclusion, repeated above, that $\psi:=\GG\zeta -\lambda x_2$ satisfies
an equation
\[
- \Delta \psi = \varphi \circ \psi \mbox{ in } \Pi
\]
which is the classical equation governing stream functions of steady planar
ideal fluid flows, so that elements of $\Sigma_\lambda$ do indeed represent
steady vortices of finite extent.

It has been noted above that Theorem \ref{thm1} applies to a wide class of
examples.
It is unfortunate that this class does not include Lamb's semicircular vortex,
because $0$ is a maximizer, relative to the weak closure of the rearrangements,
of the relevant variational problem; see \cite{GRB:Lamb}.
Lamb's vortex is a particularly interesting example because it is a
closed-form solution, and the maximizers, relative to the class of
rearrangements, are known from \cite{GRB:Lamb} to be the translates parallel
to the $x_1$-axis of a single function.

\section{Concentration-compactness and Theorem \ref{thm2}.}
Here we present a sequence of Lemmas leading to the proof of Theorem \ref{thm2}.
The first is a slight reformulation of Lions \cite[Lemma 1.1]{PLL:CC1},
and we omit the proof:

\begin{lemma} \label{lm1} {\bf (Concentration-Compactness.)}
Let $\{\xi_n\}_{n=1}^\infty$ be a sequence of non-negative elements
of $L^1(\bbR^N)$ and suppose
\[
\limsup_{n\to\infty} \int_{R^N} \xi_n  = \mu
\]
where $0\leq\mu<\infty$. Then, after passing to a subsequence, one of
the following holds:

\noindent {\rm (i) (Compactness)} There exists a sequence
$\{y_n\}_{n=1}^\infty$ in $\bbR^N$ such that
$\forall\epsilon>0$ $\exists R>0$ such that
\[
\forall n \quad \int_{y_n+B_R} \xi_n \geq \mu - \epsilon \; ;
\]
{\rm (ii) (Vanishing)}
\[
\forall R>0 \quad \lim_{n\to\infty} \sup_{y \in \bbR^N} \int_{y+B_R} \xi_n = 0 \; ;
\]
{\rm (iii) (Dichotomy)} There exists $\alpha$, $0<\alpha<\mu$, such
that for all $\epsilon>0$ and all large $n$, there exist
$\xi_n^{(1)}=1_{\Omega_n^{(1)}}\xi_n$ and
$\xi_n^{(2)}=1_{\Omega_n^{(2)}}\xi_n$, for some
disjoint measurable $\Omega_n^{(1)},\Omega_n^{(2)}\subset\bbR^N$, such
that, for all $n$,
\[
0 \leq \int_{R^N} \xi - (\xi_n^{(1)} + \xi_n^{(2)}) < \epsilon
\]
\[
-\epsilon < \int_{R^N} \xi_n^{(1)} - \alpha < \epsilon
\]
\[
-\epsilon < \int_{R^N} \xi_n^{(2)} - (\mu-\alpha) < \epsilon
\]
\[
\mbox{dist}(\Omega_n^{(1)},\Omega_n^{(2)}) \to \infty
\mbox{ as } n\to\infty \; .
\]
\end{lemma}

\noindent{\bf Remarks.}
Notice that if $\mu=0$ then the whole sequence has the Vanishing Property.

We will apply this result to maximizing sequences of $E-\lambda I$
in $\mathcal{R}_+(\xi)$, for suitable $\xi$ and $\lambda$. In this
connection, it should be noted that if $\xi \in L^2(\Pi)$ is
non-negative and has compact support then, for sequences in
$\mathcal{R}_+(\xi)$, it follows from H\"{o}lder's inequality and
equimeasurability that convergence in $\| \cdot \|_1$ and $\| \cdot
\|_2$ are equivalent.

The following alternative form of the Green's function will be useful:
\begin{equation}
G(x,y) = \frac{1}{4\pi}\log\left(
1+\frac{4x_2y_2}{(x_1-y_1)^2+(x_2-y_2)^2} \right).
\label{eq3}
\end{equation}

\bigskip

The following estimates are derived in Burton \cite[Lemmas 1 \& 5]{GRB:VP}:

\begin{lemma} \label{lm2}
Given $A>0$, we can choose positive numbers $b,c,d,e$, and $0<\beta<1$,
depending only on $A$, such that if $\xi \in L^2(\Pi)$ satisfies
$|\mbox{suppt}(\xi)| \leq A$ then we have\\
{\rm (i)}
$| \GG \xi (x_1,x_2) | \leq \| \xi \|_2 (b+c\log x_2), \quad  x_2>e$;\\
{\rm (ii)} $|\GG\xi(x_1,x_2)|\leq d|x_2|^\beta\|\xi\|_2, \quad 0<x_2<e$.
\end{lemma}

\begin{lemma} \label{lm6}
{\rm (i)}
Given $A>0$ and $2<p<\infty$ we can choose a positive number $N$ such that
$|\nabla \GG \xi (x_1,x_2)| \leq N \| \xi \|_p$ for all $\xi \in L^p(\Pi)$
vanishing outside a set of area $A$.\\
{\rm (ii)}
Given $A>0$, $Z>0$ and $2<p<\infty$, we can choose a positive number $f$
such that if $\xi \in L^p(\Pi)$  is Steiner-symmetric about the $x_2$-axis,
$\xi$ satisfies $|\mbox{suppt}(\xi)| \leq A$ and $\xi(x_1,x_2)=0$ for $x_2>Z$
then we have
\[
|\GG\xi(x_1,x_2)| \leq f \| \xi \|_p x_2 \min\{1,|x_1|^{-1/2p}\}.
\]
\end{lemma}

\medskip

\noindent
The following Lemma shows that $E(\zeta)<\infty$ provided that $\|\zeta\|_1$,
$\|\zeta\|_2$ and $I(|\zeta|)$ are all finite.
If $I(\zeta)=\infty$ we adopt the convention
$E(\zeta)-\lambda I(\zeta)=-\infty$ for $\lambda>0$.

\begin{lemma}
\label{lm9} There is a constant $C>0$ such that
\[
\|\GG\zeta\|_\infty \leq C(\|\zeta\|_1 + \|\zeta\|_2 + I(|\zeta|))
\]
for all measurable functions $\zeta$, provided that the right-hand side
is finite.
\end{lemma}

\begin{proof}
It is enough to consider the case $\zeta\geq0$. We note the
inequality
\begin{eqnarray*}
\log(a+b+c)&\leq&\log(3\max\{a,b,c\})\\ &\leq&
\log3+(\log a)_+ +(\log b)_+ +(\log c)_+
\end{eqnarray*}
for positive $a,b,c$, and write $\rho:=|x-y|$ in the formula
(\ref{eq3}) for $G$ to obtain
\[
\int_{y_2\geq{x_2}/2}\log\left(1+\frac{4x_2y_2}{\rho^2}\right)\zeta(y)\rd y
\leq \int_\Pi \log\left(1+\frac{8y_2^2}{\rho^2}\right)\zeta(y)\rd y.
\]
Now
\[
\int_\Pi \left( \log(8y_2^2) \right)_+ \, \zeta(y)\rd y \leq
\int_\Pi(\log8+2y_2)\zeta(y)\rd y=(\log8)\|\zeta\|_1+2I(\zeta),
\]
and
\begin{eqnarray*}
\int_\Pi \left( \log\rho^{-2} \right)_+ \, \zeta(y)\rd y &=&
\int_{\rho\leq1}(-2\log\rho)\zeta(y)\rd y\\
&\leq&
\left(\int_{\rho\leq1}4(\log\rho)^2\rd y\right)^{1/2}\|\zeta\|_2,
\end{eqnarray*}
hence
\[
\int_{y_2\geq{x_2}/2}\log\left(1+\frac{4x_2y_2}{\rho^2}\right)\zeta(y)\rd y
\leq \mbox{const.}(\|\zeta\|_1+I(\zeta)+\|\zeta\|_2).
\]
Also
\begin{eqnarray*}
\int_{y_2\leq{x_2}/2}\log\left(1+\frac{4x_2y_2}{\rho^2}\right)\zeta(y)\rd y
&\leq& \int_\Pi\log\left(1+\frac{2x_2^2}{x_2^2/4}\right)\zeta(y)\rd y\\
&=& (\log9)\|\zeta\|_1,
\end{eqnarray*}
and the desired inequality follows.
\end{proof}

\begin{lemma}
\label{lm10} Given positive numbers $M_1,M_2,M_3$, the functional
$E$ is Lipschitz continuous in $\|\cdot\|_2$ relative to
\[
V:=\{\zeta\in L^2(\Pi)\mid |\mbox{suppt}(\zeta)|\leq M_1, \;
I(|\zeta|)\leq M_2, \; \|\zeta\|_2\leq M_3 \}.
\]
\end{lemma}

\begin{proof}
For $\xi,\eta\in V$ we have
\begin{eqnarray*}
|E(\xi)-E(\eta)| &=& \frac{1}{2}\int_\Pi(\xi+\eta)\GG(\xi-\eta)\\
&\leq& \|\xi-\eta\|_1\|\GG(\xi+\eta)\|_\infty\\
&\leq& (2M_1)^{1/2}\|\xi-\eta\|_2
C(\|\xi+\eta\|_1+\|\xi+\eta\|_2+I(|\xi+\eta|))\\
&\leq&(2M_1)^{1/2}\|\xi-\eta\|_2 C(2M_1^{1/2}M_3+2M_3+2M_2),
\end{eqnarray*}
where $C$ is the constant provided by Lemma \ref{lm9}.
\end{proof}

\begin{lemma} \label{lm3}
Let $\zeta_0 \in L^2$ be non-negative, suppose
$|\mbox{suppt}(\zeta_0)| =\pi a^2$ for some $0<a<\infty$, and let
$\lambda > 0$. Then\\
{\rm (i)} there exists $Z>0$ (depending on $a$, $\lambda$ and
$\|\zeta_0\|_2$ only) such that, for all
$\zeta\in\mathcal{R}_+(\zeta_0)$,
\[
\GG\zeta(x_1,x_2) - \lambda x_2 <0 \quad \forall x_2>Z ;
\]
{\rm (ii)} if $\zeta \in  \mathcal{R}_+(\zeta_0)$ and $h=\zeta1_U$
where $U$ is a set on which
$\GG \zeta - \lambda x_2$
is nowhere positive, then
\[
(E-\lambda I)(\zeta -h) \geq (E - \lambda I)(\zeta)
\]
with strict inequality unless $h=0$, and in particular we can take 
$U=\bbR\times(Z,\infty)$;\\
{\rm (iii)} any maximizer of $E-\lambda I$ relative to
$\overline{\mathcal{R}(\zeta_0)^\mathrm{w}}$ is supported in
$\bbR\times[0,Z]$;\\
{\rm (iv)} if $\zeta\in \overline{\mathcal{R}(\zeta_0)^\mathrm{w}}$
with $\|\zeta\|_X<\infty$, and
$h=\zeta1_{\bbR\times(Z,\infty)}$, then there is a rearrangement
$h^\prime$ of $h$ supported in
$\bbR\times[0,Z]\backslash\mbox{suppt}(\zeta)$ such that
\[
(E-\lambda I)(\zeta-h+h^\prime)\geq(E-\lambda I)(\zeta);
\]
moreover
$\zeta-h+h^\prime$ is a rearrangement of $\zeta$.
\end{lemma}

\begin{proof}
(i) follows easily from the estimate of Lemma \ref{lm2}.

\noindent
For (ii), observe that
\begin{eqnarray*}
(E-\lambda I)(\zeta - h)
& = & (E-\lambda I)(\zeta)-\int_\Pi (\GG\zeta-\lambda x_2)h+E(h)\\
& = & (E-\lambda I)(\zeta)-\int_U (\GG\zeta-\lambda x_2)h +E(h)\\
& \geq & (E-\lambda I)(\zeta) + E(h),
\end{eqnarray*}
since on $U$ we have $(\GG\zeta-\lambda x_2) \leq 0$ and $h \geq 0$.
The result follows since $E(h)>0$ unless $h=0$.

\noindent
(iii) now follows from (ii), since if
$\zeta\in\overline{\mathcal{R}(\zeta_0)^\mathrm{w}}$ then
$\zeta-h\in\overline{\mathcal{R}(\zeta_0)^\mathrm{w}}$ also, using
results of Douglas \cite{RJD:UD}.

\noindent
(iv) is trivial if $h=0$. Suppose therefore that $h \neq 0$. Let
\[
\epsilon=(E-\lambda I)(\zeta-h)-(E-\lambda I)(\zeta)
\]
which is positive by (ii). In view of the decay of $\GG(\zeta-h)$ at
the $x_2$-axis quantified in Lemma \ref{lm2}(ii) it is enough to
form $h^\prime$ by rearranging $h$ on the part of a narrow strip
along the $x_2$-axis outside $\mbox{suppt}(\zeta)$; this is justified
since any two sets of equal finite positive Lebesgue measure are
measure-theoretically equivalent.
\end{proof}

\begin{lemma}
\label{lm8} Let $\zeta_0 \in L^2(\Pi)$ be a non-negative function
with support of finite area, and let $\lambda>0$. Then $E-\lambda I$
has the same supremum on all of the sets
$\overline{\mathcal{R}(\zeta_0)^\mathrm{w}}$,
$\mathcal{R}_+(\zeta_0)$, $\mathcal{RC}(\zeta_0)$, and
$\mathcal{R}(\zeta_0)$.
\end{lemma}

\begin{proof}
In view of the inclusions (\ref{eq14}) it will be enough to prove
equality of the suprema on the first and last sets in the list. Let
$\xi\in\overline{\mathcal{R}(\zeta_0)^\mathrm{w}}$. Then
$\xi^\prime:=\xi1_{\bbR\times(0,Z)}\in\overline{\mathcal{R}(\zeta_0)^\mathrm{w}}$
and, by Lemma \ref{lm3}(ii),
\[
(E-\lambda I)(\xi^\prime)\geq (E-\lambda I)(\xi) .
\]
By the monotone convergence theorem, given $\epsilon>0$ we can
choose $R>0$ such that $\xi\mathrm{''} :=\xi^\prime1_{(-R,R)\times \bbR}
= \xi 1_{(-R,R)\times (0,Z)}$, which also belongs to
$\overline{\mathcal{R}(\zeta_0)^\mathrm{w}}$, satisfies
\[
(E-\lambda I)(\xi\mathrm{''}) > (E-\lambda I)(\xi) - \epsilon.
\]
Now, by compactness of $\GG$ as an operator on
$L^2((-R,R)\times(0,Z))$, within every weak neighbourhood of
$\xi\mathrm{''}$ we can find
$\xi\mathrm{'''}\in\mathcal{R}_+(\zeta_0)$ supported in
$(-R,R)\times(0,Z)$ with
\[
-\epsilon<(E-\lambda I)(\xi\mathrm{'''})-(E-\lambda
I)(\xi\mathrm{''})<\epsilon.
\]
Finally, if $\xi\mathrm{'''}\not\in\mathcal{R}(\zeta_0)$, given
$\delta>0$ sufficiently small, we may choose any rearrangement
$\eta_\delta$ of $\zeta_0^\Delta - \xi^\Delta$ on a subset of
$(1/\delta,\pi a^2/\delta)\times(0,\delta)$ to find that
$\xi_\delta:=\xi\mathrm{'''}+\eta_\delta$ is a rearrangement of
$\zeta_0$ supported in a bounded subset of $\bbR\times(0,Z)$,
and that as $\delta\to 0$ we have $\xi_\delta\to\xi\mathrm{'''}$
weakly in $L^2$ and $(E-\lambda I)(\xi_\delta)\to(E-\lambda
I)(\xi\mathrm{'''})$.
\end{proof}

\begin{lemma} \label{lm4} {\bf (Vanishing excluded.)}
Suppose that $\zeta_0$, $a$, $\lambda$ and $\Sigma_\lambda$ satisfy
the hypotheses of the Maximization Theorem \ref{thm2}. Then no
maximizing sequence for $E-\lambda I$ relative to
$\mathcal{R}_+(\zeta_0)$ has the Vanishing Property of Lemma
\ref{lm1}.
\end{lemma}

\begin{proof} Suppose $\{ \zeta_n \}_{n=1}^\infty$ is a maximizing
sequence for $E-\lambda I$ relative to $\mathcal{R}_+(\zeta_0)$ that
has the Vanishing Property,
reformulated by the equivalence of $\| \cdot \|_1$ and $\| \cdot \|_2$
on $\mathcal{R}_+(\zeta_0)$ as
\[
\forall R>0 \quad \lim_{n\to\infty} \sup_{y\in{R}^N} \int_{y+B_R}
\zeta_n^2 = 0 .
\]
By Lemma \ref{lm3}(ii), we can modify the $\zeta_n$ so they are
supported in $\bbR\times[0,Z]$, while remaining a maximizing
sequence with the Vanishing Property.
Consider any $R>0$. Then for $x\in\Pi$, relative to $B:=x+B_R$, we
have
\[
\left.
\begin{array}{l}
\| \zeta_n \|_{L^2(B)} \to 0 \\
\| \zeta_n \|_{L^1(B)} \leq \mbox{const.}\|\zeta_n \|_{L^2(B)} \to 0 \mbox{ (by H\"{o}lder's inequality)} \\
I(\zeta_n 1_B) \leq Z \| \zeta_n \|_{L^1(B)} \to 0
\end{array}
\right\}
\]
uniformly over $x \in \Pi$,
so $\GG(\zeta_n 1_B)(x) \to 0$ as $n \to \infty$ uniformly over
$x\in\Pi$, by Lemma \ref{lm9}.
By writing the Green's function in the form
(\ref{eq3}) we estimate
\[
\GG\left(\zeta_n(1-1_B)\right)(x) \leq \frac{Z^2}{\pi R^2}
\|\zeta_n\|_1 \leq \frac{Z^2}{\pi R^2} \|\zeta_0\|_1 .
\]
Hence, as $n \to \infty$,
\[
E(\zeta_n) \leq \|\zeta_n\|_1 \left( \frac{Z^2}{\pi R^2}
\|\zeta_0\|_1 +o(1) \right) \leq \|\zeta_0\|_1 \left( \frac{Z^2}{\pi
R^2} \|\zeta_0\|_1 +o(1) \right).
\]
This holds for every $R>0$, hence $E(\zeta_n) \to 0$ as $n\to\infty$.
Hence
\[
\limsup_{n\to\infty} (E-\lambda I)(\zeta_n) \leq 0 .
\]
But the hypotheses of the Lemma, together with Lemma \ref{lm8}
ensure that the supremum of $E-\lambda I$ relative to
$\mathcal{R}_+(\zeta_0)$ is positive. Thus Vanishing does not occur.
\end{proof}

\begin{lemma} \label{lm5} {\bf (Dichotomy excluded.)}
Suppose that $\zeta_0$, $a$, $\lambda$, $\Sigma_\lambda$ and $S_\lambda$ satisfy
the hypotheses of the Maximization Theorem \ref{thm2}. Then no
maximizing sequence for $E-\lambda I$ relative to
$\mathcal{R}_+(\zeta_0)$ has the Dichotomy Property of Lemma
\ref{lm1}.
\end{lemma}
\begin{proof}
Suppose $\{ \zeta_n \}_{n=1}^\infty$ is a maximizing sequence for
$E-\lambda I$ relative to $\mathcal{R}_+(\zeta_0)$ that has the
Dichotomy Property. In view of the remarks on convergence following
Lemma \ref{lm1}
we can assume that, for some $0<\alpha<\mu$, and some
restrictions $\{\zeta_n^{(i)}\}_{i=1}^3$ of $\zeta_0$ to sets
$\{\Omega_n^{(i)}\}_{n=1}^3$ partitioning $\Pi$,
we have
\begin{equation}
\left.
\begin{array}{c}
\|\zeta_n^{(3)}\|_2 \to 0,\\
\|\zeta_n^{(1)}\|_2^2 \to \alpha,\\
\|\zeta_n^{(2)}\|_2^2 \to \beta:= \mu - \alpha,\\
\mbox{dist}(\mbox{suppt}(\zeta_n^{(1)}),\mbox{suppt}(\zeta_n^{(2)})) \to
\infty,\\
(E-\lambda I)(\zeta_n^{(1)}+\zeta_n^{(2)}+\zeta_n^{(3)}) \to S_\lambda,
\end{array}
\right\}
\label{eq18}
\end{equation}
as $n\to\infty$.

We may multiply the functions $\{\zeta_n^{(i)}\}_{i=1}^3$ by
$1_{\bbR\times(0,Z)}$, where $Z$ is the number provided by Lemma \ref{lm3}(i),
yet still assume the last two lines of $(\ref{eq18})$ to hold,
in view of Lemma \ref{lm3}(ii).
We also have
\begin{eqnarray*}
(E-\lambda I)(\zeta^{(1)}_n + \zeta^{(2)}_n) & = &
(E-\lambda I)(\zeta_n) -\lambda I (\zeta^{(3)}_n)\\
& & -\int_\Pi \zeta^{(3)}_n\GG(\zeta^{(1)}_n + \zeta^{(2)}_n +
{\textstyle \frac{1}{2}} \zeta^{(3)}_n) \\
& \to & S_\lambda \quad \mbox{as } n\to\infty
\end{eqnarray*}
by Lemma \ref{lm9}, so we may replace $\zeta_n^{(3)}$ by $0$ and suppose
$\zeta_n = \zeta_n^{(1)} + \zeta_n^{(2)}$ for all $n$.

Formula (\ref{eq3}) for the Green's function leads to the estimate
\begin{eqnarray*}
\int_\Pi \zeta_n^{(1)} \GG \zeta_n^{(2)} &  \leq & \pi^{-1}\|\zeta_n^{(1)}\|_1
\|\zeta_n^{(2)}\|_1 Z^2
\mbox{dist}(\mbox{suppt}(\zeta_n^{(1)}),\mbox{suppt}(\zeta_n^{(2)}))^{-2}\\
& \to & 0 \quad \mbox{as } n\to\infty.
\end{eqnarray*}
Consequently
\begin{eqnarray}
(E-\lambda I)(\zeta_n^{(1)})+(E-\lambda I)(\zeta_n^{(2)})
& = &   (E-\lambda I)(\zeta_n) - \int_\Pi \zeta_n^{(1)} \GG\zeta_n^{(2)} \nonumber \\
& \to & S_\lambda
\quad \mbox{as } n\to\infty.
\label{eq4}
\end{eqnarray}

Let $\zeta_n^{{(1)}*}$, $\zeta_n^{{(2)}*}$ denote the Steiner
symmetrizations of $\zeta_n^{(1)}$, $\zeta_n^{(2)}$ about the $x_2$-axis.
Then we have
\begin{equation}
(E-\lambda I)(\zeta_n^{(i)*})
\geq (E-\lambda I)(\zeta_n^{(i)}), \quad i=1,2,
\label{eq20}
\end{equation}
by Riesz's rearrangement inequality in conjunction with formula (\ref{eq3})
for $G$, and the symmetrization-invariance of $I$.

From the estimate of Lemma \ref{lm6}(ii) (in case $p=\infty$ replacing $p$ by
any $2<p<\infty$), there is a positive
number $k$ such that $\GG\zeta(x)-\lambda x_2>0$ only if $x\in (-k,k)
\times (0,Z)$, uniformly over Steiner symmetric $\zeta \in
\mathcal{R}_+(\zeta_0)$.
Let
\[
\zeta_n^{(i)**}=\zeta_n^{(i)*}1_{\{x\mid \GG\zeta_n^{(i)*}(x)-\lambda x_2
>0\}}, \quad i=1,2.
\]
Then from (\ref{eq20}) and Lemma \ref{lm3}(ii) we have
\begin{equation}
(E-\lambda I) (\zeta_n^{(i)**}) \geq (E-\lambda I)(\zeta_n^{(i)*})
\geq (E-\lambda I)(\zeta_n^{(i)}), \quad i=1,2.
\label{eq19}
\end{equation}
Moreover we have
\[
\mbox{suppt}(\zeta_n^{(i)**})  \subset  \{ x \mid -k<x_1<k \}, \quad i=1,2.
\]
Now define
\[
\zeta_n^{{(1)}***}(x) = \zeta_n^{{(1)}**}(x_1-k,x_2),
\]
\[
\zeta_n^{(2)***}(x) = \zeta_n^{(2)**}(x_1+k,x_2),
\]
which are supported in $(0,2k) \times (0,Z)$ and $(-2k,0) \times
(0,Z)$ respectively. Then, from (\ref{eq19}),
\[
(E-\lambda I)(\zeta_n^{(i)***}) = (E-\lambda I)(\zeta_n^{(i)**})
\geq (E-\lambda I)(\zeta_n^{(i)}), \quad i=1,2.
\]
From the definitions and \eqref{eq14} we have
$\zeta_n^{(1)***} + \zeta_n^{(2)***} \in \mathcal{R}_+(\zeta_0) \subset \overline{\mathcal{R}(\zeta_0)^\mathrm{w}}$ and
after passing to a subsequence we can assume that
$\zeta_n^{{(1)}***} \to \overline{\zeta^{(1)}}$ and
$\zeta_n^{(2)***} \to \overline{\zeta^{(2)}}$, say, weakly in $L^2$, and so
$\overline{\zeta} := \overline{\zeta^{(1)}} + \overline{\zeta^{(2)}}
\in \overline{\mathcal{R}(\zeta_0)^\mathrm{w}}$.
Now, using compactness of $\GG$ as an operator on $L^2((-2k,2k)\times(0,Z))$,
from (\ref{eq4}) we have
\begin{eqnarray*}
(E-\lambda I)(\overline{\zeta^{(1)}}) &+&
(E-\lambda I)(\overline{\zeta^{(2)}}) \\
&=&
\lim_{n\to\infty} \left( (E-\lambda I)(\zeta_n^{{(1)}***}) +
(E-\lambda I)(\zeta_n^{{(2)}***}) \right) \\
&\geq&
S_\lambda
\end{eqnarray*}
and therefore
\begin{eqnarray}
(E-\lambda I)(\overline{\zeta}) &=&
(E-\lambda I) (\overline{\zeta^{(1)}}) + (E-\lambda I)(\overline{\zeta^{(2)}})
 + \int_\Pi\overline{\zeta^{(1)}}\GG\overline{\zeta^{(2)}} \nonumber\\
&\geq& (E-\lambda I)(\overline{\zeta^{(1)}})+(E-\lambda I)(\overline{\zeta^{(2)}})
\label{eq21} \\
&\geq&S_\lambda . \nonumber
\end{eqnarray}
Therefore $\overline{\zeta} \in \Sigma_\lambda$ so, by hypothesis,
$\overline{\zeta} \in \mathcal{R}(\zeta_0)$.
It follows that
\[
\| \overline{\zeta^{(1)}} \|_2^2 + \| \overline{\zeta^{(2)}} \|_2^2 =
\| \overline{\zeta} \|^2 = \mu.
\]
Since
$\| \overline{\zeta^{(1)}} \|_2^2 \leq \alpha$ and
$\| \overline{\zeta^{(2)}} \|_2^2 \leq \beta$,
we deduce 
$\| \overline{\zeta^{(1)}} \|_2^2 = \alpha$ and
$\| \overline{\zeta^{(2)}} \|_2^2 = \beta$,
so both $\overline{\zeta^{(1)}}$ and $\overline{\zeta^{(2)}}$ are non-zero.
Hence
\[
\int_\Pi \overline{\zeta^{(1)}} \GG \overline{\zeta^{(2)}} > 0.
\]
Therefore strict inequality holds in (\ref{eq21})
contradicting the definition of $S_\lambda$.
\end{proof}

\bigskip

\noindent {\bf Proof of Theorem \ref{thm2}.}
There is no loss of generality in supposing $2<p<\infty$.

To prove (i), we first consider a
maximizing sequence $\{\zeta_n\}_{n=1}^\infty$ for $E-\lambda I$
comprising elements of $\mathcal{R}_+(\zeta_0)$ and having
the Compactness Property.
There is thus a sequence
$\{y_n\}_{n=1}^\infty$ in $\Pi$ such that
\begin{equation}
\forall \epsilon>0 \; \exists R>0 \mbox{ s.t. } \forall n \;
\|\zeta_n1_{\Pi\setminus(y_n+B_R)}\|_2^2 < \epsilon.
\label{eq17}
\end{equation}
We assume the $y_n$ to lie on the $x_2$ axis, which results in no loss of
generality in view of the invariance of $E-\lambda I$ under translations
in the $x_1$-direction.

Let $\zeta_n^0 = \zeta_n 1_{\bbR \times (0,Z)}$ and
$\zeta_n^R = \zeta_n 1_{(-R,R) \times (0,Z)}$,
where $Z$ is the number given in Lemma \ref{lm3}(i) and $R>0$ is arbitrary.
Then $\{ \zeta_n^0 \}_{n=1}^\infty$ is also a maximizing sequence by
Lemma \ref{lm3}(ii), and (\ref{eq17})
has the consequence that
\begin{equation}
\forall \epsilon>0 \; \exists R>0 \mbox{ s.t. } \forall n \;
\|\zeta_n^0 - \zeta_n^R\|_2^2 < \epsilon.
\label{eq22}
\end{equation}

After passing to a subsequence, we can suppose that
$\{ \zeta_n^0 \}_{n=1}^\infty$ converges weakly in $L^2(\Pi)$
to a limit $\zeta^0$, hence
$\zeta_n^R \to \zeta^R := \zeta^0 1_{(-R,R) \times (0,Z)}$
weakly as $n \to \infty$.
With this notation (\ref{eq22}) takes the form
\begin{equation}
\|\zeta_n^0 - \zeta_n^R\|_2 \to 0
\mbox{ as } R\to\infty, \mbox{ uniformly over } n.
\label{eq8}
\end{equation}
Now $E(\zeta_n^R) \to E(\zeta^R)$ as $n\to\infty$, for each $R$.
We have
\[
E(\zeta_n^0) = E(\zeta_n^R+(\zeta_n^0-\zeta_n^R))
= E(\zeta_n^R)+E(\zeta_n^0-\zeta_n^R)
+ \int_\Pi\zeta_n^R \GG(\zeta_n^0-\zeta_n^R),
\]
whence
\begin{equation}
|E(\zeta_n^0)-E(\zeta_n^R)|\leq \mbox{const.}\|\zeta_n^0-\zeta_n^R\|_2
\label{eq9}
\end{equation}
by Lemma \ref{lm10}, and the constant is independent of $R$ and $n$.
Now from (\ref{eq8}) and (\ref{eq9}) we have
\begin{equation}
|E(\zeta_n^0)-E(\zeta_n^R)| \to 0 \mbox{ as } R\to\infty,
\mbox{ uniformly over } n.
\label{eq7}
\end{equation}

But $E(\zeta_n^R)\to E(\zeta^R)$ as $n\to\infty$ for each fixed $R$ by weak
continuity, and $E(\zeta^R)\to E(\zeta^0)$ as $R\to\infty$ by
the Monotone Convergence Theorem, and in conjunction with (\ref{eq7})
this yields $E(\zeta_n^0) \to E(\zeta^0)$.

We have
\begin{equation}
|I(\zeta_n^0)-I(\zeta_n^R)| \leq R\|\zeta_n^0-\zeta_n^R\|_1. \label{eq10}
\end{equation}
Let $\epsilon>0$. Now
\[
|I(\zeta_n^0)-I(\zeta^0)| \leq |I(\zeta_n^0)-I(\zeta_n^R)| +
|I(\zeta_n^R)-I(\zeta^R)| + |I(\zeta^R)-I(\zeta^0)|;
\]
we may choose $R>0$ to make the first term less than $\epsilon/3$ for all $n$
by (\ref{eq8}) and (\ref{eq10}), and the last term less than
$\epsilon/3$ for all $n$ by the Monotone Convergence Theorem, then the middle
term is less than $\epsilon/3$ for all sufficiently large $n$ by
weak convergence. Hence $I(\zeta_n^0) \to I(\zeta^0)$ as $n\to\infty$.

Thus $(E-\lambda I)(\zeta_n^0) \to (E-\lambda I)(\zeta^0)$ as
$n\to\infty$, so $(E-\lambda I)(\zeta^0)=S_\lambda$. Therefore
$\zeta^0 \in \mathcal{R}(\zeta_0)$ by hypothesis, so
$\zeta_n^0 \to \zeta^0$ strongly in $L^2(\Pi)$ by uniform convexity.

Since $\zeta_n^0$ and $\zeta_n-\zeta_n^0$ are supported on disjoint sets,
\[
\|\zeta_n-\zeta_n^0\|_2^2 = \|\zeta_n\|_2^2-\|\zeta_n^0\|_2^2 \leq
\|\zeta_0\|_2^2-\|\zeta_n^0\|_2^2 \to 0,
\]
so $\zeta_n \to \zeta^0$ as $n\to\infty$.
We have thus proved that a compact maximizing sequence has a subsequence
which, after suitable translations in the $x_1$-direction, converges
strongly in $L^2(\Pi)$ to an element of $\Sigma_\lambda$.
Now Lemmas \ref{lm1}, \ref{lm4} and \ref{lm5} show that every
maximizing sequence contains a subsequence having the Compactness
Property, and (i) follows.

To prove (ii), let $\{\zeta_n\}_{n=1}^\infty$  be a maximizing
sequence for $E-\lambda I$ relative to
$\overline{\mathcal{R}(\zeta_0)^\mathrm{w}}$ comprising elements of
$\mathcal{R}_+(\zeta_0)$. Suppose that
$\mbox{dist}_2(\zeta_n,\Sigma_\lambda)\not\to0$ as $n\to\infty$.
Then, after discarding a subsequence, we can suppose
\begin{equation}
\mbox{dist}_2(\zeta_n,\Sigma_\lambda)>\delta\quad\forall n,
\label{eq11}
\end{equation}
where $\delta$ is some positive number.

But, by (i), $\{\zeta_n\}_{n=1}^\infty$ can be replaced by a
subsequence that, after suitable translations in the
$x_1$-direction, converges in $\|\cdot\|_2$ to an element of
$\Sigma_\lambda$. This contradicts the supposition (\ref{eq11}),
showing that $\mbox{dist}_2(\zeta_n,\Sigma_\lambda)\to0$ as
$n\to\infty$, proving (ii).

To prove (iii) observe that Lemma \ref{lm8} shows the existence
of maximizing sequences in $\mathcal{R}_+(\zeta_0)$, which can, by
(i), be assumed to contain subsequences converging to elements of
$\Sigma_\lambda$, which is therefore non-empty.

Finally, to prove (iv) observe that, for fixed $x_2$ and $y_2$,
formula \eqref{eq3} shows that $G(x,y)$ is a positive strictly decreasing
function of $|x_1-y_1|$ alone, so we can apply the one-dimensional case of
Lieb's analysis \cite[Lemma 3]{Lieb:Choq} of equality in Riesz's rearrangement
inequality, on pairs of lines parallel to the $x_1$-axis, to deduce that
every $\zeta \in \Sigma_\lambda$ is, after a translation, Steiner-symmetric
about the $x_2$-axis.
From Lemma \ref{lm3}(ii) we know that every  $\zeta \in \Sigma_\lambda$
is supported in the set where $\GG\zeta(x)-\lambda x_2 >0$, which is bounded
by Lemma \ref{lm2}(i) and Lemma \ref{lm6}(ii).
The functional relationship $\zeta=\varphi\circ(\GG\zeta-\lambda x_2)$,
where $\zeta\in \Sigma_\lambda$ is any element and $\varphi$ is some
({\em a priori} unknown) increasing function, is given by
\cite[Theorem 16(i)]{GRB:VP} (where it forms the first variation
condition at a maximum).
\hfill \qedsymbol

\bigskip

\section{Existence, Transport and Theorem \ref{thm1}.}
\begin{lemma}
\label{lm7}
Let $0 \leq \zeta_0 \in L^p(\Pi)$, for some $2<p\leq\infty$ and suppose
$| \mbox{suppt}(\zeta_0) | = \pi a^2$ for some $0<a<\infty$.
For $0 < \lambda < \infty$ let $S_\lambda$ be the supremum
of $E-\lambda I$ relative to
$\overline{\mathcal{R}(\zeta_0)^\mathrm{w}}$ and let
$\Sigma_\lambda$ be the set of maximizers of $E-\lambda I$ relative
to $\overline{\mathcal{R}(\zeta_0)^\mathrm{w}}$. Then, there exists
$\Lambda>0$ such that, if $0<\lambda<\Lambda$, then $\emptyset \neq
\Sigma_\lambda \subset \mathcal{R}(\zeta_0)$.
\end{lemma}

\begin{proof}
There is no loss of generality in supposing $2<p<\infty$.
Since $E$ is unbounded above on $\mathcal{R}(\zeta_0)$, which may be
seen by translating $\zeta_0$ away to infinity in the
$x_2$-direction, we have $S_\lambda \to\infty$ as $\lambda \to 0$.
Therefore we can choose $\Lambda>0$ such that, if
$0<\lambda<\Lambda$ then $S_\lambda>M$, where $M$ is a positive
number to be chosen later.

Consider $\lambda$ with $0<\lambda<\Lambda$, and consider
$\overline{\zeta}\in\Sigma_\lambda$.
Arguing as in Douglas \cite[Theorem 4.1]{RJD:UD}, the strict convexity of
$E-\lambda I$ ensures that $\overline{\zeta}$ is an extreme point of
$\overline{\mathcal{R}(\zeta_0)^w}$ so
$\overline{\zeta} \in \mathcal{RC}(\zeta_0) \subset \mathcal{R}_+(\zeta_0)$
by \cite[Theorem 2.1]{RJD:UD}.
Let $m:=\sup_{x\in\Pi} (\half \GG\overline{\zeta}(x)-\lambda x_2)$.
Then
\[
M<S_\lambda = \int_\Pi
\overline{\zeta}(x)\left(\frac{1}{2}\GG\overline{\zeta}(x)-\lambda x_2\right)\rd x
\leq m\|\overline{\zeta}\|_1 \leq m\|\zeta_0\|_1,
\]
so $m>M/\|\zeta_0\|_1$.
From Lemma \ref{lm6}(i) we have
\[
|\nabla \GG\overline{\zeta}(x)| \leq N\|\overline{\zeta}\|_p \quad \forall x\in\Pi 
\]
where $N$ is a positive constant independent of $\lambda$ and $M$, hence if
$x\in\Pi$ is such that $\frac{1}{2}\GG\overline{\zeta}(x)-\lambda x_2 > m/2$, and
$y\in\Pi$ is such that $|y-x|<2a$ and $y_2<x_2$, then
\[
{\textstyle
\frac{1}{2}\GG\overline{\zeta}(y) - \lambda y_2 >
\frac{1}{2}\left(\GG\overline{\zeta}(x)-2aN\|\zeta_0\|_1\right)-\lambda x_2
}
\]
\[
> \frac{m}{2} -2aN\| \zeta_0\|_1
> \frac{M}{2\|\zeta_0\|_1} - 2aN\|\zeta_0\|_1 > 0,
\]
provided we choose $M=5aN\|\zeta_0\|_1^2$; note this shows $x_2>a$
because $\frac{1}{2}\GG\overline{\zeta}(y)-\lambda y_2$ vanishes when
$y_2=0$. Hence
\[
{\textstyle
|\{y\in\Pi\mid \GG\overline{\zeta}(y)-\lambda y_2>0\}| >
|\{y\in\Pi\mid\frac{1}{2}\GG\overline{\zeta}(y)-\lambda y_2>0\}|
>2\pi a^2.
}
\]
It follows that if $\overline{\zeta}$ is a proper curtailment of a rearrangement
of $\zeta_0$, then we have the freedom to  choose $\zeta_1$ supported in
$\{y\in\Pi\mid
\GG\overline{\zeta}(y)-\lambda{y_2}>0\}\setminus\mbox{suppt}(\overline{\zeta})$
such that $\overline{\zeta}+\zeta_1\in\mathcal{R}(\zeta_0)$, and then
\[
(E-\lambda I)(\overline{\zeta}+\zeta_1) = (E-\lambda I)(\overline{\zeta}) +
E(\zeta_1) + \int_\Pi(\GG\overline{\zeta}-\lambda x_2)\zeta_1
> (E-\lambda I)(\overline{\zeta}),
\]
contradiction. So every maximizer belongs to $\mathcal{R}(\zeta_0)$.
\end{proof}

\medskip

Recall the definition of an $L^p$-regular solution given in Definition \ref{LpRegSol}.

\begin{lemma} \label{Renormalized}
Let $2<p<\infty$ and let $\zeta$ be an $L^p$-regular solution of the vorticity equation.
Let $\psi(t,x)=\GG\zeta(t,x)-\lambda x_2$ and set
$u(t,x)=\nabla^\perp\psi(t,x)-\lambda e_1$.
Suppose $\omega_0 \in L^p(\Pi)$.
Then the initial value problem for the linear transport equation
\begin{equation}
\begin{cases}
\partial_t \omega + \mathrm{div}(\omega u) =0 &\\
\omega(0) =\omega_0 &
\end{cases}
\label{eq25}
\end{equation}
has a unique weak solution $\omega \in L^\infty_\mathrm{loc}([0,\infty),L^p(\Pi))$.
Moreover, $\omega \in C([0,\infty),L^p(\Pi))$ and $\omega$ satisfies
the renormalisation property in the form $\omega(t) \in \mathcal{R}(\omega_0)$
for almost all $t > 0$.
\label{lm12}
\end{lemma}

\begin{proof}
We begin by extending $\psi$ to the whole of $\bbR^2$ as a function odd in
$x_2$, which is accomplished by allowing arbitrary $x \in \bbR^2$ in the
formula \eqref{eq2}; then extending $\zeta$ to $\bbR^2$ as a function odd in
$x_2$ gives $-\Delta \psi = \zeta$ throughout $\bbR^2$ and we take
$u=\nabla^\perp \psi + \lambda e_1$ throughout $\bbR^2$.
Similarly we extend $\omega_0$ to $\bbR^2$ as a function odd in $x_2$.
Write $\psi_0(x)=\psi(x)+\lambda x_2$.

Now $\| \psi_0(t,\cdot) \|_\infty$ is bounded by Lemma \ref{lm9} and it then
follows from elliptic regularity theory, specifically Agmon
\cite[Thm. 6.1]{SAg}, that $\| \psi_0(t,\cdot) \|_{2,p,B}$ is bounded over all
unit balls $B$ uniformly over every bounded interval of $t$.
Hence $\| u(t,\cdot) \|_{\infty}$ is bounded on every bounded interval of $t$.

The DiPerna-Lions theory of transport equations
\cite[Thm. II.2, Cor. II.1 and Cor. II.2]{RJDiP:PLL}
assures us of the existence of a unique solution
$\omega \in L^\infty_\mathrm{loc}([0,\infty),L^p(\bbR^2))$ to \eqref{eq25},
that $\omega \in C([0,\infty),L^p(\bbR^2))$ and that $\omega$
has the renormalisation property
\[
\partial_t (\beta\circ\omega) + \mathrm{div}((\beta\circ\omega)u) =0
\]
for every $\beta \in C^1(\bbR)$ that satisfies
$|\beta^\prime(s)| \leq \mathrm{const.}(1+|s|^{p/2})$.
Moreover $\omega$ is odd in $x_2$ from the uniqueness.

Now consider a test function of the form $\chi(t)\varphi_R(x)$ where
$\chi \in  \mathscr{D}(\bbR)$ and $\varphi_R \in \mathscr{D}(\bbR^2)$
satisfies $0 \leq \varphi_R \leq 1$ everywhere, $\varphi_R(x)=1$ if $|x|<R$,
$\varphi_R(x)=0$ if $|x|>2R$ and $|\nabla\varphi_R| <2/R$ everywhere.
Then, for any $\beta$ as above,
\begin{equation}
\int_{\bbR^2}\int_{\bbR} \chi^\prime(t) \varphi_R(x) \beta(\omega(t,x))\rd t\rd x 
+ \int_{\bbR^2}\int_{\bbR}
\chi(t) \nabla\varphi_R(x) \cdot u(t,x) \beta(\omega(t,x))\rd t\rd x
= 0 .
\label{eq26}
\end{equation}
We now suppose further that $|\beta(s)| \leq \mathrm{const.}|s|^p$ for all $s$,
choose $\sigma>2$ such that $1/2+1/p+1/\sigma=1$ and deduce
\begin{eqnarray*}
\left| \int_{\bbR^2} \beta(\omega(t,x)) \nabla\varphi_R(x) \cdot u(t,x) \rd x \right|
& \leq &
\|\beta(\omega(t,\cdot)\|_p \| u(t,\cdot) \|_2 \|\nabla \varphi_R \|_\sigma \\
& \leq &
\mathrm{const.}\|\omega(t,\cdot)\|_p \| u(t,\cdot) \|_2 R^{2/\sigma -1}\\
& \to & 0 \mbox{ as } R \to \infty
\mbox{ uniformly over } t \in \mathrm{suppt} \chi.
\end{eqnarray*}
From \eqref{eq26} we now have
\[
\int_{\bbR^2} \beta(\omega(t,x)) \rd x = \mathrm{const.}
\]
for each $\beta$; taking $\beta$ to be a mollification of $1_{[\alpha,\infty)}$
for $\alpha>0$ we deduce that the positive part of $\omega(t)$ is a
rearrangement of the positive part of $\omega_0$ and similarly for
the negative parts.
\end{proof}

\noindent
{\bf Remark} Note that it follows from Lemma \ref{lm12}, in particular, that $\|\zeta(t,\cdot)\|_p$
and $\|\zeta(t,\cdot)\|_1$ are conserved if $\zeta$ is an $L^p$-regular
solution.

We also observe that a version of Lemma \ref{Renormalized}, in the case
$\lambda = 0$, was established in \cite[Proposition 1]{MLN:Enstrf}.

\noindent {\bf Proof of Theorem \ref{thm1}.}
There is no loss of generality in supposing $2<p<\infty$.
Choose $Z>0$ such that
$\GG\omega(x)-\lambda x_2<0$ for $x_2>Z$ provided
\begin{equation}
\left.
\begin{array}{rcl}
\omega &\geq& 0\\
|\mbox{suppt}(\omega)| &<& A\\
\|\omega\|_2 &\leq& \|\zeta_0\|_2+1 \\
I(\omega) &\leq& \sup I(\Sigma_\lambda)+1
\end{array}
\right\}
\label{eq23}
\end{equation}
by Lemma \ref{lm9}.
Write
\[
\widetilde{\omega} = \omega 1_{\bbR\times(0,Z)}.
\]

Recall the notation $S_{\lambda}$ introduced in Theorem \ref{thm2}, as the supremum of $E-\lambda I$ relative to $\overline{\mathcal{R}(\zeta_0)^\mathrm{w}}$.

Then we have
\[
(E-\lambda I)(\widetilde{\omega}) \geq (E-\lambda I)(\omega)
\]
and
\[
(E-\lambda I)(\omega) \to S_\lambda \quad \mbox{as} \quad
\mbox{dist}_Y(\omega,\Sigma_\lambda) \to 0,
\]
for $\omega$ satisfying (\ref{eq23}), by Lemma \ref{lm10}.

Suppose, to seek a contradiction, that
$\{\omega^n(\cdot)\}_{n=1}^\infty$ are non-negative $L^p$-regular solutions of
the vorticity equation (\ref{eq1}) for which
\[
\mbox{dist}_Y(\omega^n(0),\Sigma_\lambda)\to0
\]
as $n\to\infty$, but
\begin{equation}
\sup_{t>0}\mbox{dist}_2(\omega^n(t),\Sigma_\lambda)
> \theta \quad \forall n,
\label{eq15}
\end{equation}
where $\theta>0$. For each $n$ choose $t_n>0$ such that
\begin{equation}
\mbox{dist}_2(\omega^n(t_n),\Sigma_\lambda) >\theta,
\label{eq16}
\end{equation}
and choose $\zeta^n_0 \in \Sigma_\lambda$ such that
\[
\| \zeta^n_0 - \omega^n(0) \|_2 \to 0 \mbox { as } n \to \infty.
\label{eq24}
\]
In view of Theorem \ref{thm2}(i), after translating in the $x_1$-direction,
and passing to a subsequence, we may additionally suppose that 
$\{ \zeta^n_0 \}_{n=1}^\infty$ converges in $L^2(\Pi)$ to a limit
in $\Sigma_\lambda$. Re-assigning the label $\zeta_0$ we shall suppose
\[
\zeta^n_0 \to \zeta_0 \in \Sigma_\lambda \mbox{ as } n \to \infty
\mbox{ in } \| \cdot \|_2 .
\] 

Now
\begin{equation}
\inf_{t>0} (E-\lambda I)(\widetilde{\omega}^n(t)) \geq S_\lambda -o(1)
\mbox{ as } n\to\infty. \label{eq13}
\end{equation}
Let $\zeta(\cdot)=\zeta^n(\cdot)$ be the solution of the transport
equation
\[
\begin{cases}
\partial_t \zeta + \mathrm{div}( \zeta u ) = 0, & \\
u = \lambda e_1  + \nabla^\perp \GG\omega^n &
\end{cases}
\]
with initial data $\zeta^n_0$. Then, continuing to use ~$\widetilde{}$~ for
restriction to $\bbR \times (0,Z)$, we have
\begin{eqnarray*}
|I(\widetilde{\zeta}^n(t))-I(\widetilde{\omega}^n(t))| & \leq &
Z\|\widetilde{\zeta}^n(t)-\widetilde{\omega}^n(t)\|_1\\
& \leq & Z\|\zeta^n_0-\omega^n(0)\|_1\\
& \leq & Z(\pi a^2+A)^{1/2}\|\zeta^n_0-\omega^n(0)\|_2
\end{eqnarray*}
and
\[ \|\widetilde{\zeta}^n(t)-\widetilde{\omega}^n(t)\|_2 \leq
\|\zeta^n(t)-\omega^n(t)\|_2 = \|\zeta^n_0-\omega^n(0)\|_2,
\]
hence
\begin{eqnarray}
|E(\widetilde{\zeta}^n(t))-E(\widetilde{\omega}^n(t))| & \leq &
\mbox{const.}\|\widetilde{\zeta}^n(t)-\widetilde{\omega}^n(t)\|_2 \nonumber\\
& \leq & \mbox{const.}\|\zeta^n_0-\omega^n(0)\|_2. \label{eq12}
\end{eqnarray}
by Lemma \ref{lm10}. It now follows from (\ref{eq13}) and
(\ref{eq12}) that $\{\widetilde{\zeta}^n(t_n)\}_{n=1}^\infty$ is a
maximizing sequence for $E-\lambda I$ relative to
$\overline{\mathcal{R}(\zeta_0)^\mathrm{w}}$. It follows from
Theorem \ref{thm2} that
\[
\mbox{dist}_2(\widetilde{\zeta}^n(t_n),\Sigma_\lambda)\to0
\]
as $n\to\infty$.

From this it follows that
\[
\|\widetilde{\zeta}^n(t_n)-\zeta^n(t_n)\|_2\to0,
\]
since the functions $\widetilde{\zeta}^n(t_n) - \zeta^n(t_n)$ and
$\widetilde{\zeta}^n(t_n)$ have disjoint supports and are therefore
orthogonal in $L^2$, so
\[
\|\widetilde{\zeta}^n(t_n)-\zeta^n(t_n)\|_2^2 =
\|\zeta^n(t_n)\|_2^2-\|\widetilde{\zeta}^n(t_n)\|_2^2 =
\|\zeta_0\|_2^2-\|\widetilde{\zeta}^n(t_n)\|_2^2 \to 0.
\]
Hence
\[
\mbox{dist}_2(\zeta^n(t_n),\Sigma_\lambda) \to 0.
\]
Since
\[
\|\zeta^n(t_n)-\omega^n(t_n)\|_2 = \|\zeta^n_0-\omega^n(0)\|_2 \to 0
\]
we deduce
\[
\mbox{dist}_2(\omega^n(t_n),\Sigma_\lambda) \to 0,
\]
and this contradicts the choices of $\theta$ and $t_n$ made in
(\ref{eq15}) and (\ref{eq16}).
\hfill\qedsymbol

\section{Further Remarks.}
If we only consider perturbations formed by adding non-negative
vorticity to a maximizer then we can prove the following variant
of Theorem \ref{thm1} concerning stability in $\|\cdot\|_Y$:

\begin{theorem}
\label{thm3} Let $\zeta_0$ be a non-negative function whose support
has finite positive area $\pi a^2$ ($a>0$) in the half-plane $\Pi$,
suppose $\zeta_0 \in L^p(\Pi)$ for some $2<p\leq \infty$, and suppose
$\lambda>0$. Let $\Sigma_\lambda$ denote the set of maximizers of
$E-\lambda I$ relative to
$\overline{\mathcal{R}(\zeta_0)^\mathrm{w}}$, and suppose
$\emptyset\neq\Sigma_\lambda\subset\mathcal{R}(\zeta_0)$. Then
$\Sigma_\lambda$ is orbitally stable, in the sense that, for every
$\epsilon> 0$ and $A>\pi a^2$, there exists $\delta>0$ such that, if
$\omega(0)$ satisfies $\omega(0)\geq\zeta_0$ for some element of
$\Sigma_\lambda$, again denoted $\zeta_0$, and if 
$\mbox{dist}_Y(\omega(0),\Sigma_\lambda)<\delta$ and
$|\mbox{suppt}(\omega(0))|<A$, then for all $t\in\bbR$, then we have
$\mbox{dist}_Y(\omega(t),\Sigma_\lambda)<\epsilon$, whenever
$\omega(t)$ denotes an $L^p$-regular solution of (\ref{eq1}) with initial data
$\omega(0)$.
\end{theorem}

\begin{proof}
We indicate the modifications that should be made to the proof of
Theorem \ref{thm1}.
We have
$I(\zeta^n(t_n))-I(\omega^n(t_n))\leq 0$. Therefore
\begin{eqnarray*}
S_\lambda - (E-\lambda I)(\omega^n(t_n) &\geq& (E-\lambda
I)(\zeta^n(t_n)-(E-\lambda I)(\omega^n(t_n))\\
&\geq& E(\zeta^n(t_n))-E(\omega^n(t_n)).
\end{eqnarray*}
Now $E(\zeta^n(t_n))-E(\omega^n(t_n)) \to 0$ by Lemma \ref{lm10},
whereas, by conservation of the impulse and energy of $\omega^n(t)$,
we have
\[
(E-\lambda I)(\omega^n(t_n))=(E-\lambda I)(\omega^n(0)) \to
S_\lambda,
\]
using Lemma \ref{lm10}.

What is now required is to choose a maximizer $\sigma_n$ close to
$\omega^n(t_n)$ in $\|\cdot\|_2$, and then after taking a
subsequence, and suitably translating the $\sigma_n$ in the
$x_1$-direction to $\sigma_n^\prime$, obtain convergence in
$\|\cdot\|_2$, say to $\sigma_0$. Then $E(\omega^n(t_n))\to
E(\sigma_0)$ and
\[
(E-\lambda I)(\omega^n(t_n))\to S_\lambda=(E-\lambda I)(\sigma_0)
\]
so $I(\omega^n(t_n)\to I(\sigma_0)$. For large $n$, we then have a
contradiction to the choice of $\theta$ and $t_n$ in (\ref{eq14})
and (\ref{eq15}), which in this case would have been
\[
\sup_{t>0}\mbox{dist}_Y(\omega^n(t),\Sigma_\lambda)
>\theta, \quad \mbox{dist}_Y(\omega^n(t_n),\Sigma_\lambda) >\theta.
\]
Hence $\sup_t\mbox{dist}_Y(\omega(t),\Sigma_\lambda) \to 0$ as
$\mbox{dist}_Y(\omega(0),\Sigma_\lambda) \to 0$, as desired.
\end{proof}

We conclude this article with some final remarks. Although we describe
our result as a nonlinear stability theorem, it falls short of what one would
desire because we only show that, for a class of steady vortex pairs $\zeta_0$, the 
perturbed trajectories stay close to the set $\Sigma_{\lambda}(\zeta_0)$,
but not necessarily to the orbit of the unperturbed steady wave 
$\{\zeta_0(\cdot-(\lambda t,0)), t \in \mathbb{R} \}$. In consequence,
the most natural problem raised by this work is to further investigate the structure of 
$\Sigma_{\lambda}$.
Another issue that bears further scrutiny is that the notions of closeness employed for the 
perturbation of initial vorticity and the change in the evolving vorticity are slightly different.
An extension of this work in any way which would include Lamb's
circular vortex-pair (the one case where we have a precise characterization of
$\Sigma_{\lambda}$) would also be very interesting.   

\vspace{1cm}

{\small {\it Acknowledgments.} The first author expresses his appreciation
of the hospitality of the University of Campinas where this work was begun.
The second author was partially supported by CNPq grant \# 303089/2010-5 and
the third author was partially supported by CNPq grant \# 306331/2010-1.
This work was supported in part by FAPESP grants \# 06/51079-2 and
\# 07/51490-7.}

\bibliographystyle{plain}

\begin{thebibliography}{10}



\bibitem{SAg}
{\sc S. Agmon.}
\newblock {The $L_p$ approach to the Dirichlet problem}.
\newblock {\em Ann. Scuola Norm. Sup. Pisa Cl. Sci. (3)}.
{\bf 13}, 405--448 (1959). 

\bibitem{VIA:COND}
{\sc V.I. Arnol$^\prime$d.}
\newblock {Conditions for nonlinear stability of stationary plane curvilinear
flows of an ideal fluid}.
\newblock {\em Soviet Math. Doklady} {\bf 162}, 773--777 (1965).
\newblock {Translation of:}
\newblock {\em Dokl. Akad. Nauk SSSR} {\bf 162}, 975--998 (1965).

\bibitem{TBB}
{\sc T. Brooke-Benjamin.}
\newblock {The alliance of practical and analytical insights into the
nonlinear problems of fluid mechanics.}
\newblock {\em Applications of Methods of Functional Analysis to Problems
in Mechanics.}
\newblock {Lecture Notes in Mathematics {\bf 503} pp. 8--29.
Springer-Verlag, Berlin, 1976.}

\bibitem{GRB:MCF}
{\sc G.R. Burton.}
\newblock {Rearrangements of functions, maximization of convex
functionals, and vortex rings}.
\newblock {\em {Math. Ann.}} {\bf 276}, 225--253 (1987).

\bibitem{GRB:VP}
{\sc G.R. Burton.}
\newblock {Steady symmetric vortex pairs and rearrangements}.
\newblock {\em {Proc. Roy. Soc. Edinburgh Sect. A}}. {\bf 108}, 269--290 (1988).

\bibitem{GRB:Lamb}
{\sc G.R. Burton.}
\newblock {Isoperimetric properties of Lamb's circular vortex-pair}.
\newblock {\em {J. Math. Fluid Mech.}} {\bf 7}, S68--S80 (2005).

\bibitem{GRB:stab}
{\sc G.R. Burton.}
\newblock {Global nonlinear stability for ideal fluid flow in
bounded planar domains}.
\newblock{\em {Arch. Rational Mech. Anal.}} {\bf 176}, 149--163 (2005).


\bibitem{D:Z} 
{\sc G. Deem and N. Zabusky.} 
\newblock {Stationary {\em V}-states, interactions, recurrence and breaking}. 
\newblock {In: \em Solitons in
Action. Lonngren, Scott, (eds). Proceedings of the Workshop in Solitons, Redstone Arsenal.} 1977, 
\newblock { NewYork, Academic Press.}  277--293.

\bibitem{RJDiP:PLL}
{\sc R.J. DiPerna and P.-L. Lions.}
\newblock {Ordinary differential equations, transport theory and
Sobolev spaces}.
\newblock {\em Invent. Math.} {\bf 98}, 511--547 (1989)

\bibitem{RJD:UD}
{\sc R.J. Douglas.}
\newblock {Rearrangements of functions on unbounded domains}.
\newblock {\em {Proc. Roy. Soc. Edinburgh Sect. A}}
{\bf 124}, 621--644 (1994).


\bibitem{D:S}
{\sc J. Duc and J. Sommeria. }
\newblock {Experimental characterization of steady two dimensional vortex couples.}
\newblock {\em J. Fluid Mech.} {\bf 192}, 175--192 (1988).

\bibitem{Lamb}
{\sc H. Lamb.}
\newblock {Hydrodynamics}.
\newblock {Cambridge Univ. Press, 6th edition, 1932.}

\bibitem{Lieb:Choq}
{\sc E.H. Lieb}
\newblock {Existence and uniqueness of the minimising solution of
Choquard's nonlinear equation.}
\newblock {\em Studies in Appl. Math.} {\bf 57}, 93--105 (1977).
  
\bibitem{PLL:CC1}
{\sc P.-L. Lions.}
\newblock {The concentration-compactness principle in the calculus of
variations. The locally compact case. I}.
\newblock {\em Ann. Inst. H. Poincar\'{e} Anal. Non Lin\'{e}aire}
{\bf 1}, 109--145 (1984).

\bibitem{LNT00} 
{\sc M.C. Lopes Filho, H.J. Nussenzveig Lopes and E. Tadmor.}
\newblock{Approximate solutions of the incompressible
Euler equations with no concentrations.}
\newblock{\em Ann. Inst. Henri Poincar\'{e}, Analyse non lin\'eaire} {\bf 17}, 371--412 (2000).

\bibitem{L2X01}
{\sc M.C. Lopes Filho, H.J. Nussenzveig Lopes and Zhouping Xin.}
\newblock{Existence of vortex sheets with reflection symmetry in two space dimensions}
\newblock{\em Arch. Rational Mech. Anal.} {\bf 158}, 235--257 (2001).

\bibitem{AJM:ALB}
{\sc A.J. Majda and A.L. Bertozzi.}
\newblock {Vorticity and Incompressible Flow.}
\newblock {Cambridge University Press, Cambridge 2002.}

\bibitem{CM:MP}
{\sc C. Marchioro and M. Pulvirenti.}
\newblock {Some considerations on the nonlinear stability of
stationary planar Euler flows.}
\newblock {\em Commun. Math. Phys.} {\bf 100}, 343--354 (1985).


\bibitem{MLN:Enstrf}
{\sc A.L. Mazzucato, M.C. Lopes Filho and H.J. Nussenzveig Lopes.}
\newblock {Weak solutions, renormalized solutions and enstrophy defects in 2D turbulence}.
\newblock {\em Arch. Rational Mech. Anal.} {\bf 179}, 353--387 (2006).

\bibitem{Moffatt}
{\sc H.K. Moffatt.} 
\newblock {Structure and stability of solutions of the Euler equations: A Lagrangian approach.}
\newblock {\em Philos. Trans. Roy. Soc. London Ser. A} {\bf  333}, 321--342 (1990).

\bibitem{Norbury}
{\sc J. Norbury.}
\newblock {Steady planar vortex pairs in an ideal fluid}. 
\newblock {\em Comm. Pure Appl. Math.} {\bf 28}, 679--700 (1975).
 
\bibitem{oz82} 
{\sc E. Overman and N. Zabusky.}
\newblock {Coaxial scattering of Euler equation translating {\em V}-states via contour dynamics}.
\newblock {\em J. Fluid Mech.} {\bf 125}, 187--202 (1982).


\bibitem{Pierrehumbert}
{\sc R. Pierrehumbert.} 
\newblock {A family of steady, translating vortex pairs with distributed vorticity}. 
\newblock {\em J. Fluid Mech.} {\bf 99}, 129--144 (1980).

\bibitem{Pocklington}
{\sc H.C. Pocklington.}
\newblock {The configuration of a pair of equal and opposite hollow and
straight vortices of finite cross-section, moving steadily through fluid.}
\newblock {\em Proc. Cambridge Philos. Soc.} {\bf 8}, 178--187 (1895).


\bibitem{pullin92} 
{\sc D. Pullin.}
\newblock {Contour Dynamics Methods.}
\newblock {\em Annual review of fluid mechanics.} {\bf 24}, 89--115 (1992).

\bibitem{saffman95} 
{\sc P. G. Saffman.}
\newblock {Vortex dynamics}. 
\newblock {Cambridge Monographs on Mechanics and Applied Mathematics.}
Cambridge University Press, New York, 1992.

\bibitem{Kelvin}
{\sc W. Thomson (Lord Kelvin).}
\newblock {Maximum and minimum energy in vortex motion.}
\newblock {\em Nature} {\bf 22}, no. 574, 618--620 (1880);
\newblock {\em Mathematical and Physical Papers 4}, 172--183,
Cambridge University Press, Cambridge, 1910.

\bibitem{Yang91} 
{\sc Jianfu Yang}
\newblock {On the existence of steady planar vortices}
\newblock {\em Ann. Univ. Ferrara Sez. VII (N. S.)} {\bf 37}, 111--129 (1991).

\bibitem{YK94} 
{\sc J. Yang and T. Kubota}
\newblock {The steady motion of a symmetric, finite core size, counterrotating
vortex pair.}
\newblock {\em SIAM J. Appl. Math.} {\bf 54}, 14--25 (1994).

\bibitem{YHW:MP}
{\sc Y.H Wan and M. Pulvirenti.}
\newblock {Nonlinear stability of circular vortex patches.}
\newblock {\em Commun. Math. Phys.} {\bf 99}, 435--450 (1985).

\bibitem{VIY}
{\sc V.I. Yudovich.}
\newblock {Non-stationary flow of an ideal incompressible liquid}.
\newblock {\em U.S.S.R. Comput. Math. and Math. Phys.} {\bf 3}, 1407-1456 (1963).
\newblock {Translation of:}
\newblock {\em Zh. Vychisl. Mat. i Mat. Fiz.} {\bf 6}, 1032--1066 (1963).


\end{thebibliography}

\noindent
{\large\bf Authors' addresses.}

\noindent
G.R Burton\\
Department of Mathematical Sciences\\
University of Bath\\
Claverton Down\\
Bath BA2 7AY\\
United Kingdom.\\
E-mail: G.R.Burton@bath.ac.uk
\bigskip

\noindent
H.J. Nussenzveig Lopes and M.C. Lopes Filho\\
Instituto de Matem\'{a}tica\\
Universidade Federal do Rio de Janeiro\\
Cidade Universit\'aria -- Ilha do Fund\~ao\\
Caixa Postal 68530\\
21941-909 Rio de Janeiro, RJ -- BRASIL.\\
E-mail: hlopes@im.ufrj.br (H.J. Nussenzveig Lopes)\\
mlopes@im.ufrj.br (M.C. Lopes Filho).
\end{document}